\documentclass[a4paper, 10pt, reqno]{article}
\usepackage{amsfonts}
\usepackage{amsmath}
\usepackage{amsthm}
\usepackage{amssymb}
\usepackage{oldgerm}
\usepackage{fancyhdr}
\usepackage{fancybox}
\usepackage{mathrsfs}
\usepackage{epsfig}
\theoremstyle{plain}
\def\phi{\varphi }
\swapnumbers

\newenvironment{enumeratei}{\begin{list} {(\roman{enumi})}%
{\usecounter{enumi}%
\setlength{\topsep}{0mm}%
\setlength{\partopsep}{0mm}%
\setlength{\itemsep}{0mm}%
\setlength{\labelsep}{2mm}%
\settowidth{\labelwidth}{(viii)}%
\setlength{\leftmargin}{1mm}%
\addtolength{\leftmargin}{\labelwidth}%
\addtolength{\leftmargin}{\labelsep}%
\setlength{\itemindent}{0mm}%
}}{\end{list}}

\newenvironment{enumeratea}{\begin{list}{(\alph{enumi})}%
{\usecounter{enumi}%
\setlength{\topsep}{0mm}%
\setlength{\partopsep}{0mm}%
\setlength{\itemsep}{0.5mm}%
\setlength{\labelsep}{2mm}%
\setlength{\labelwidth}{5mm}%
\setlength{\leftmargin}{0mm}%
\addtolength{\leftmargin}{\labelwidth}%
\addtolength{\leftmargin}{\labelsep}%
\setlength{\itemindent}{0mm}%
}}{\end{list}}

\theoremstyle{plain}
\newtheorem{Thm}{Theorem}[section]

\newtheorem{La}[Thm]{Lemma}
\newtheorem{Prop}[Thm]{Proposition}

\theoremstyle{definition}
\newtheorem{Def}[Thm]{Definition}

\theoremstyle{remark}
\newtheorem{Rem}[Thm]{Remark}

\newtheorem{Ex}[Thm]{Example}

\numberwithin{equation}{section}

\newcommand{\liea}{\mathfrak{a}}

\newcommand{\CC}{\mathbb{C}}
\newcommand{\RR}{\mathbb{R}}

\newcommand{\ZZ}{\mathbb{Z}}

\newcommand{\Waff}{W_{\textnormal{aff}}}

\newcommand{\la}{\lambda_{\alpha}}
\newcommand{\alcove}{A_0}

\renewcommand{\Im}{\textnormal{Im} \,}

\begin{document}
\title{The heat semigroup in the compact Heckman-Opdam setting and the Segal-Bargmann transform}
\author{Heiko Remling and Margit R\"osler\\
Institut f\"ur Mathematik, TU Clausthal\\
Erzstra\ss e 1\\
D-38678 Clausthal-Zellerfeld, Germany\\
remling@math.tu-clausthal.de, roesler@math.tu-clausthal.de}
\date{}
\maketitle

\begin{abstract} In the first part of this paper,  we study the heat equation and the heat kernel associated with the
Heckman-Opdam Laplacian in the compact, Weyl-group invariant  setting. In particular, this Laplacian gives
 rise to a Feller-Markov semigroup on a fundamental alcove of the affine Weyl group. The second part of the paper is devoted to
 the Segal-Bargmann transform in our context. A Hilbert space of holomorphic functions is defined such that the $L^2$-heat transform becomes a unitary isomorphism.
 \end{abstract}



\begin{large}\end{large}

%
\section{Introduction}
%

Heckman-Opdam theory provides a powerful generalization of the theory of noncompact and compact Riemannian symmetric spaces and their spherical functions. In this theory (see e.g.  \cite{Op95}, \cite{Op00}, \cite{HS94}), the system of invariant differential operators on a
Riemannian symmetric space is replaced by a commuting algebra of differential reflection operators, called Dunkl operators, which depend on some root system and on multiplicity parameters on the roots. The joint spectral problem for these operators is solved by multivariable hypergeometric functions and hypergeometric polynomials which include the
spherical functions of Riemannian symmetric spaces for certain discrete values of the multiplicities.

In the context of Dunkl operators, the heat equation has already been studied in various settings. The rational case was treated by R\"osler in \cite{Ro98}, while Schapira \cite{Scha} studied the heat equation in the noncompact Heckman-Opdam theory. In the present paper we investigate the compact symmetric case, where we assume invariance under the Weyl group $W.$
We are concerned with the heat equation for the Heckman-Opdam Laplacian $L_m$ on a closed fundamental alcove $\alcove$ for the
affine Weyl group. This Laplacian generalizes the Laplace-Beltrami operator on a Riemannian symmetric space of compact type. We prove that $L_m$ has a closure which generates a Feller-Markov semigroup on the alcove, the Heckman-Opdam heat semigroup. We study smoothness properties of the heat kernel and also develop an $L^p$-theory for the heat equation on $\alcove.$

The second main topic of this paper is the Segal-Bargmann transform. Several generalizations of the classical Segal-Bargmann transform to different settings are known. The Segal-Bargmann transform for compact Lie groups was introduced by Hall \cite{Hall},
where also the case of compact symmetric spaces was considered. Different approaches in the case of compact symmetric spaces were given by Stenzel \cite{St99} and Faraut \cite{Faraut}. In the framework of Dunkl theory, the rational case has been studied by several authors, see \cite{Soltani},  \cite{SaidOrsted}, \cite{Sontz}.
Apart from the rank one case, an explicit description of
the Segal-Bargmann space as an $L^2$-space of holomorphic functions has so far not been found in this setting.
Ben Sa\"{\i}d and \O{}rsted \cite{SaidOrsted} instead gave a description as a Fock space generated by a certain reproducing kernel (which is given by the rational Dunkl kernel).
The noncompact, symmetric Heckman-Opdam case was investigated in 2007 by \'{O}lafsson and Schlichtkrull \cite{OS07}.

In this paper, we study the Segal-Bargmann transform in the compact symmetric Heckman-Opdam setting.
We extend the heat transform to a unitary isomorphism from the weighted $L^2$-space on the alcove $\alcove$ to a Segal-Bargmann space $\mathcal H_t$, which is a Hilbert space of holomorphic functions. Its inner product is described as an $L^2$-product involving the heat kernel from the noncompact theory as a weight.

The organization of this paper is as follows: In Section \ref{sec_dunkl}, we recall some basics of trigonometric Dunkl theory. In Section \ref{sec_heat} the heat equation and the heat semigroup on the fundamental alcove $\alcove$ are studied. Finally, the Segal-Bargmann transform is developed in Section \ref{sec_SB}.

%
\section{Fundamentals of Trigonometric Dunkl Theory}\label{sec_dunkl}
%

We start with a short review of the fundamentals of trigonometric Dunkl theory which will be needed in this article. For details, we refer to the work of Heckman and Opdam \cite{Op95}, \cite{Op00}, \cite{HS94}, and the references cited there.

Let $\liea$ be a finite-dimensional Euclidean space with inner product $\langle \cdot, \cdot \rangle$, which is extended to a complex bilinear form on the complexification $\liea_{\CC}$ of $\liea$. We identify $\liea$ with its dual space $\liea^* = \text{Hom} (\liea , \RR)$ via the given inner product.
Let $\Sigma \subset \liea$ be a (not necessarily reduced) root system.
For $\alpha \in \Sigma$ we write $\alpha^{\vee} := 2 \alpha/\langle \alpha, \alpha \rangle$ for the coroot of $\alpha$ and denote by $s_{\alpha} (x) = x - \langle \alpha^{\vee}, x \rangle \alpha\,$ the reflection in the hyperplane $H_{\alpha}$ perpendicular to $\alpha$.

The reflections $\{ s_{\alpha} \, : \, \alpha \in \Sigma\}$ generate the Weyl group $W=W(\Sigma)$. We define the root lattice $Q := \ZZ.\Sigma$ and the coroot lattice $Q^{\vee} = \ZZ.\Sigma^{\vee}$. Further, we fix some positive subsystem $\Sigma^+$ of $\Sigma$.
An element $\lambda \in \liea$ is called (strictly) dominant, if $\langle \lambda, \alpha \rangle \ge 0$ (respectively $>0$) for all $\alpha\in \Sigma^+$. We write
\[\liea^+ := \{ \lambda \in \liea \, : \, \langle \lambda, \alpha^{\vee} \rangle > 0 \, \, \forall \alpha \in \Sigma^+ \}\]
 for the Weyl chamber of strictly dominant elements.

For $\alpha \in \Sigma$ and $\lambda \in \liea_{\CC}$ let
\begin{equation*}
\lambda_{\alpha} := \frac{ \langle \lambda, \alpha \rangle}{\langle \alpha, \alpha \rangle}.
\end{equation*}
The \emph{weight lattice} is given by
\begin{equation*}
\Lambda := \{ \lambda \in \liea \, : \, \lambda_{\alpha} \in \ZZ \,\, (\forall \alpha \in \Sigma) \}
\end{equation*}
and the set
\begin{equation*}
\Lambda^+ := \{ \lambda \in \liea \, : \, \lambda_{\alpha} \in \ZZ^+ \,\, (\forall \alpha \in \Sigma^+) \}
\end{equation*}
is called the lattice of dominant weights. Here we use the notation $\ZZ^+ := \{0,1,2, \ldots \}$. The positive root lattice $Q^+ = \ZZ^+. \Sigma^+$ defines a partial ordering $\preceq$ on $\liea$:
\[
\mu \preceq \lambda \iff \lambda - \mu \in Q^+.
\]
This ordering is called the dominance ordering. Two simple properties are given in the following Lemma.

\begin{La}\label{Ordnung_W}\begin{enumerate}\itemsep=-1pt
\item[\rm{(i)}]
Let $\gamma \in \overline{\liea^+}$ be dominant. Then $w \gamma \preceq \gamma$ for all $w \in W$.
\item[\rm{(ii)}] Let $\lambda, \mu \in \Lambda^+$ be dominant weights with $\mu \preceq \lambda$. Then $|\mu| \le |\lambda|$.
\end{enumerate}
\end{La}

\begin{proof}
Part (i) is Lemma 10.3B in \cite{Humph}. For the proof of (ii), notice that
$\lambda + \mu$ is also dominant and $\lambda - \mu$
is a sum of positive roots. Therefore
\[
0 \le \langle \lambda + \mu, \lambda - \mu \rangle = |\lambda|^2 - |\mu|^2.
\]
\end{proof}

A \emph{multiplicity function} is a $W$-invariant map $m: \Sigma \to \CC$, $\alpha \mapsto m_{\alpha}$. We denote the set of multiplicity functions by $\mathcal M$. In this article we only consider non-negative multiplicities, i.e. $m_{\alpha} \ge 0$ for all $\alpha \in \Sigma$. Define
\[\rho = \rho(m) := \frac{1}{2} \sum_{\alpha \in \Sigma^+} m_{\alpha} \alpha.\]
\begin{Def}\label{Def_DO}
Let $\xi \in \liea_{\CC}$ and $m \in \mathcal M$. The \emph{Dunkl-Cherednik operator} associated with $\Sigma$ and $m$ is given by
\[
T_{\xi} = T(\xi,m) := \partial_{\xi} + \sum_{\alpha \in \Sigma^+} m_{\alpha} \langle \alpha, \xi \rangle \frac{1}{1-e^{-2 \alpha}} (1- s_{\alpha}) - \langle \rho , \xi \rangle,
\]
where $\partial_{\xi}$ is the usual directional derivative and $e^{\lambda}(\xi) := e^{\langle \lambda, \xi \rangle}$ for $\lambda, \xi \in \liea_{\CC}$.
\end{Def}

\begin{Rem}\label{Bem_Not}
Heckman and Opdam use a slightly different notation. They consider a root system $R$ with multiplicity $k$, which is connected to our notation via
\[
R = 2 \Sigma, \quad k_{2 \alpha} = \frac{1}{2} m_{\alpha}.
\]
Our notation comes from the theory of symmetric spaces.
\end{Rem}

For fixed multiplicity $m$, the operators $T_{\xi}$, $\xi \in \liea_{\CC}$ commute.
Therefore the assignment $\xi \mapsto T(\xi,m)$ uniquely extends to a homomorphism on the symmetric algebra $S(\liea_{\CC})$ over $\liea_{\CC}$, which may be identified with the algebra of complex polynomials on $\liea_{\CC}$. Let $T(p,m)$ be the operator which in this way corresponds to $p \in S(\liea_{\CC})$. If $p\in S(\liea_{\CC})^W$, the subspace of $W$-invariant polynomials on $\liea_{\CC}$, then $T(p,m)$ acts as a differential operator on the space of $W$-invariant analytic functions on $ \liea$.

The solution of the joint spectral problem for these differential operators is due to Heckman and Opdam, see \cite{HS94} and \cite{Op95}:

\begin{Thm}
For each fixed spectral parameter $\lambda \in \liea_{\CC}$, the so-called hypergeometric system
\[
T(p,m) \phi = p(\lambda) \phi \quad \text{ for all }\, p \in S(\liea_{\CC})^W
\]
has a unique $W$-invariant solution $\phi =F_{\lambda}(m;\cdot) = F(\lambda,m; \cdot)$ which is analytic on $\liea$ and satisfies $F_\lambda(m;0)=1$.
Moreover, there is a $W$-invariant tubular neighborhood $U$ of $\liea$ in $\liea_{\CC}$ such that $F$ extends to a (single-valued) holomorphic function $F: \liea_{\CC} \times \mathcal M^{\mathrm{reg}}\times U \to \CC$.
\end{Thm}

The function $F(\lambda,m;x)$ is $W$-invariant in both $\lambda$ and $x$. It is called the \emph{hypergeometric function} associated with $\Sigma$.
For certain spectral parameters $\lambda$, the functions $F_\lambda$ are actually trigonometric polynomials, the so-called Heckman-Opdam polynomials. In order to make this precise, we need some more notation.

Let $\mathcal T := \text{lin} \{e^{i\lambda} \, : \, \lambda \in \Lambda \}$ be the space of trigonometric polynomials associated with $\Lambda$. Trigonometric polynomials are $\pi Q^{\vee}$-periodic, and $T_{\xi} \mathcal T \subset \mathcal T$.
Consider the torus $T = \liea / \pi Q^{\vee}$ with the $W$-invariant weight function
\[
w_m:= \prod_{\alpha \in \Sigma^+} \left| e^{i\alpha} - e^{- i\alpha} \right|^{m_{\alpha}}.
\]

\noindent Let
\[ M_\lambda := \sum_{\mu \in W.\lambda} e^{i\mu}, \quad \lambda \in \Lambda^+ \]
denote the $W$-invariant orbit sums. They form a basis of the space of $W$-invariant trigonometric polynomials $\mathcal T^W$.
For $\lambda \in \Lambda^+$ the \emph{Heckman-Opdam polynomials} associated with $\Sigma$ are defined by
\[
 P_{\lambda} =P_\lambda(m; \cdot) := \sum_{\mu \in \Lambda^+, \, \mu \preceq \lambda} c_{\lambda \mu}(m) M_\mu
\]
where the coefficients $c_{\lambda \mu}(m)$ are uniquely determined by the conditions

\begin{enumeratei}
 \item $c_{\lambda \lambda}(m) = 1$
 \item $P_\lambda$ is orthogonal to $M_\mu$ in $L^2(T; w_m)$ for all $\mu \in \Lambda^+$ with $ \mu \prec \lambda$.
\end{enumeratei}

\smallskip
The Jacobi polynomials $P_\lambda$ form an orthogonal basis of $L^2(T,w_m)^W$, the subspace of $W$-invariant elements from $L^2(T,w_m)$.

\begin{Rem}\label{Bem_i}
Notice that our notation slightly differs from that of Heckman and Opdam (e.g. \cite{HS94}, \cite{Op00}), namely by a factor $i$ in the spectral variable. This choice of notation will be more convenient for our purposes.
\end{Rem}

The connection between the Jacobi polynomials and the hypergeometric function is as follows:

\begin{La}\label{Hyper_Jacobi} (See \cite{HS94})
For all $z \in \liea_{\CC}$ and $\lambda \in \Lambda^+$,
\[
P_\lambda(m ; z)= c(\lambda + \rho,m)^{-1} F_{\lambda+\rho}(m ; iz),
\]
where the $c$-function $c(\lambda + \rho,m) = P_\lambda(m;0) ^{-1}$ is given by
\[
c(\lambda + \rho,m) = \prod_{\alpha \in \Sigma^+} \frac{\Gamma(\la + \rho_{\alpha} + \frac{1}{4} m_{\alpha/2}) \Gamma(\rho_{\alpha} + \frac{1}{4} m_{\alpha/2} + \frac{1}{2} m_{\alpha})}{ \Gamma(\lambda_{\alpha} + \rho_{\alpha} + \frac{1}{4} m_{\alpha/2} + \frac{1}{2} m_{\alpha}) \Gamma(\rho_{\alpha} + \frac{1}{4} m_{\alpha/2})}.
\]
\end{La}

We shall work with the renormalized Jacobi polynomials, defined by
 \[R_{\lambda} (z) :=R_\lambda(m;z) := c(\lambda + \rho,m) P_{\lambda}(m;z) = F_{\lambda+\rho}(m;iz).\]
They satisfy
\[R_{\lambda} (0) = 1.\]
Dividing the torus $T = \liea / \pi Q^{\vee}$ by the action of the Weyl group $W$ gives the closed fundamental alcove
\[\alcove = \{ x\in \liea: 0 \leq \langle \alpha, x\rangle \leq \pi \quad (\forall \alpha \in \Sigma^+)\}.\]

 We may consider $W$-invariant trigonometric polynomials as functions on $\alcove$. Another way of considering a $W$-invariant and $\pi Q^{\vee}$-periodic function $f$ on $\liea$ is to say that $f$ is $\Waff$-invariant, where
\[
\Waff = \pi Q^{\vee} \rtimes W
\]
is the affine Weyl group. The closed alcove $\alcove$ is a fundamental domain for the action of $\Waff$ on $\liea$.
 
The Jacobi polynomials $R_{\lambda}$ are orthogonal with respect to the inner product
\[
\langle f,g \rangle_m = \int_{\alcove} f(x)\overline{g(x)} w_m(x) \, dx,
\]
but they are not orthonormal. We put
\begin{equation*}
r_{\lambda} := \frac{1}{\| R_{\lambda} \|^2_m}.
\end{equation*}
Then the set $\{ \sqrt{r_{\lambda}} R_{\lambda} \, : \, \lambda \in \Lambda^+ \}$ is an orthonormal basis of $L^2(\alcove, w_m)$.

\begin{Rem}\label{Norm_Jacobi} We shall need the following facts about the Jacobi polynomials $P_\lambda$ and $R_\lambda$:
\begin{enumeratea}
 \item The $L^2(\alcove)$-norm of $P_\lambda$ is given by
\[
\| P_{\lambda} \|_m^2 = \prod_{\alpha \in \Sigma^+} \frac{\Gamma( \lambda_{\alpha} + \rho_{\alpha} - \frac{1}{4} m_{\alpha/2} - \frac{1}{2} m_{\alpha} + 1)}{\Gamma( \lambda_{\alpha} + \rho_{\alpha} - \frac{1}{4} m_{\alpha/2} + 1)} \cdot \frac{\Gamma(\lambda_{\alpha} + \rho_{\alpha} + \frac{1}{4} m_{\alpha/2} + \frac{1}{2} m_{\alpha})}{\Gamma(\lambda_{\alpha} + \rho_{\alpha} + \frac{1}{4} m_{\alpha/2})}
\]
see Theorem 3.5.5 in \cite{HS94}. (Notice that $P_\lambda$ is $W$-invariant.)
 \item The coefficients $c_{\lambda \mu}(m)$ of the $P_{\lambda}$ are rational functions in $m_{\alpha}, \,\alpha \in \Sigma^+.$ Moreover, their numerator and denominator polynomials have nonnegative integral coefficients. This was observed in
\cite{Mac87}, Par. 11. As a consequence, the renormalized polynomial
$R_{\lambda}$ is, for non-negative $m$, a convex combination of exponentials $e^{i\gamma}$:
\[
R_{\lambda} = \sum_{\substack{\gamma \in W.\mu \\ \mu \in \Lambda^+, \; \mu \preceq \lambda}} d_{\lambda \gamma} e^{i \gamma}
\]
with coefficients $d_{\lambda \gamma}\ge 0$ and $\sum_{\gamma} d_{\lambda \gamma} =1$.
 \item Because of $\overline{e^{i \langle \mu, x \rangle}} = e^{ -i \langle \mu, x \rangle}$ we have
\begin{equation*}
\overline{R_{\lambda} (x)} = R_{\lambda} (-x), \qquad x \in \liea,
\end{equation*}
and more general for $z \in \liea_{\CC}$:
\begin{equation*}
 \overline{R_{\lambda} (- \overline z)} = R_{\lambda} (z), \qquad \overline{R_{\lambda} (\overline z)} = R_{\lambda} (-z).
\end{equation*}
\end{enumeratea}
\end{Rem}

%
\section{The heat equation on the alcove}\label{sec_heat}
%

In this section we consider the $W$-invariant part of the Heckman-Opdam Laplacian on the alcove, which coincides with the radial part of the Laplace Beltrami operator of a compact symmetric space $U/K$ in geometric cases. We study the associated heat semigroup - the \emph{Heckman-Opdam heat semigroup} - and its integral kernel $\Gamma_m$.
In particular, we show that this heat kernel can be holomorphically extended to $\liea_{\CC} \times \liea_{\CC}$, which will be important for the following section, where we study the Segal-Bargmann transform.

The \emph{Heckman-Opdam Laplacian} is defined by
\begin{equation*}
\Delta_m := \sum_{i=1}^q T(\xi_i,m)^2 - |\rho|^2
\end{equation*}
where $T(\xi_i,m)$ is the Dunkl-Cherednik operator of Definition \ref{Def_DO} and $\{\xi_1, \ldots \xi_q \}$ is an orthonormal basis of $\liea$. The operator $\Delta_m$ does not depend on the choice of the basis and has the explicit form
\begin{equation*}
\Delta_m f(x) = \Delta f(x) + \sum_{\alpha \in \Sigma^+} m_{\alpha} \coth \langle \alpha, x \rangle \partial_{\alpha} f(x) - \sum_{\alpha \in \Sigma^+} \frac{m_{\alpha}|\alpha|^2}{2 \sinh^2 \langle \alpha, x \rangle} (f(x) - f(s_{\alpha} x))
\end{equation*}
where $\Delta$ denotes the Euclidean Laplacian on $\liea$ (See \cite{Scha} and recall $R= 2 \Sigma$ and $k_{2\alpha} = \frac 1 2 m_{\alpha}$). \\

We now restrict our attention to $W$-invariant functions. Keeping in mind that our notation differs by a factor $i$ from that of Heckman and Opdam (see Remark \ref{Bem_i}),
 we consider the operator
\begin{equation*}
L_m:= \Delta + \sum_{\alpha \in \Sigma^+} m_{\alpha} \cot \langle \alpha, x \rangle \partial_{\alpha}
\end{equation*}
on $C^2(\liea)^W$. Then for $f(x) = g(ix)$, we just have
\begin{equation}\label{I}
L_m f(x) = - (\Delta_m g)(ix).
\end{equation}

\begin{Rem}\label{Rem_Beltrami}
Consider a compact symmetric space $U/K$ on which $K$ acts from the left with restricted root system $\Sigma$ and geometric multiplicity $m$. Then $L_m$ is just the radial part of the Laplace-Beltrami operator on $U/K$. See Proposition 3.11, Chapter II in \cite{Hel84}.

\end{Rem}

\smallskip

The Jacobi polynomials $R_{\lambda}$ are eigenfunctions of $L_m$:
\begin{equation}\label{eigenequation}
L_m R_{\lambda} = - \langle \lambda, \lambda + 2\rho \rangle R_{\lambda}, \quad \lambda \in \Lambda^+.
\end{equation}
This follows from equation ($\ref{I}$). The eigenvalues are negative,
\[
- \langle \lambda, \lambda + 2\rho \rangle = - |\lambda|^2 - 2 \langle \lambda, \rho \rangle \le 0
\]
since $\lambda$ and $\rho$ are both contained in the Weyl chamber $\overline{\liea^+}$ and therefore $\langle \lambda, \rho \rangle \ge 0$. We shall use the abbreviation
\begin{equation*}
 \theta_{\lambda} := \langle \lambda, \lambda + 2\rho \rangle.
\end{equation*}

%

The Heckman-Opdam heat equation on $\alcove$ is given by
\begin{equation*}
L_m u = \partial_t u.
\end{equation*}
A formal derivation via Heckman-Opdam transform
\[
\widehat f (\lambda) := \int_{\alcove} f(x) R_{\lambda} (-x) w_m(x) \, dx
\]
motivates the following

 \begin{Def}\label{Kern}
The \emph{heat kernel} $\Gamma_m$ on $\alcove \times \alcove \times (0, \infty)$ is defined by
 \begin{equation*}
 \Gamma_m(x,y,t) := \sum_{\lambda \in \Lambda^+} r_{\lambda} e^{- \theta_\lambda t} R_{\lambda}(x) R_{\lambda}(-y).
 \end{equation*}
 \end{Def}

We shall also consider $\Gamma_m$ as a function on $\liea\times \liea\times (0,\infty)$ which is $\Waff$-invariant in the first and second argument. We still have to show that the series converges. This will be a consequence of the following Lemma which states that the growth of the $r_{\lambda}$ is polynomial in $\lambda_{\alpha}, \,\alpha \in \Sigma^+.$ We start with some simple observations. First, recall from Remark \ref{Norm_Jacobi} (b) that the coefficients $d_{\lambda\mu}$ in the exponential expansion of the Jacobi polynomials $R_\lambda $ are
nonnegative and sum up to $1$. Therefore
\begin{equation}\label{unibound}
 |R_{\lambda} (x) | \le R_{\lambda} (0) = 1 \quad (\forall x \in \liea).
\end{equation}
For the summands of $\Gamma_m$ this implies
\begin{align}\label{Schranke_Summanden}
\left| r_{\lambda} e^{- \theta_{\lambda} t} R_{\lambda}(x) R_{\lambda}(-y)\right| \le r_{\lambda} e^{- \theta_{\lambda} t}.
\end{align}

\begin{La}\label{La_q}
There exists a constant $C > 0$ such that
\[
| r_{\lambda} | \le C \cdot \prod_{\alpha \in \Sigma^+, \, \lambda_{\alpha} \ne 0} \lambda_{\alpha}^{m_{\alpha}}.
\]
\end{La}

\begin{proof}
According to Lemma \ref{Hyper_Jacobi} and Remark \ref{Norm_Jacobi} (a) we have
\begin{align*}
 r_{\lambda} = & \frac 1 {\|R_{\lambda} \|_m^2} = \frac{1}{\| c(\lambda + \rho) P_{\lambda} \|_m^2} \\
= & \left( \prod_{\alpha \in \Sigma^+} \frac{ \Gamma(\lambda_{\alpha} + \rho_{\alpha} + \frac{1}{4} m_{\alpha/2} + \frac{1}{2} m_{\alpha}) \Gamma(\rho_{\alpha} + \frac{1}{4} m_{\alpha/2})}{\Gamma(\la + \rho_{\alpha} + \frac{1}{4} m_{\alpha/2}) \Gamma(\rho_{\alpha} + \frac{1}{4} m_{\alpha/2} + \frac{1}{2} m_{\alpha})} \right)^2 \cdot \\
& \cdot \prod_{\alpha \in \Sigma^+} \frac{\Gamma( \lambda_{\alpha} + \rho_{\alpha} - \frac{1}{4} m_{\alpha/2} + 1)}{\Gamma( \lambda_{\alpha} + \rho_{\alpha} - \frac{1}{4} m_{\alpha/2} - \frac{1}{2} m_{\alpha} + 1)} \cdot \frac{\Gamma(\lambda_{\alpha} + \rho_{\alpha} + \frac{1}{4} m_{\alpha/2})}{\Gamma(\lambda_{\alpha} + \rho_{\alpha} + \frac{1}{4} m_{\alpha/2} + \frac{1}{2} m_{\alpha})} \\
=& \,c\cdot \prod_{\alpha \in \Sigma^+} f_\alpha(\lambda_\alpha)
\end{align*}
where $c>0$ is a constant depending only on $m$ and
\[
f_{\alpha} (\lambda_{\alpha}) = \frac{ \Gamma(\lambda_{\alpha} + \rho_{\alpha} + \frac{1}{4} m_{\alpha/2} + \frac{1}{2} m_{\alpha}) \Gamma( \lambda_{\alpha} + \rho_{\alpha} - \frac{1}{4} m_{\alpha/2} + 1)}{ \Gamma(\la + \rho_{\alpha} + \frac{1}{4} m_{\alpha/2}) \Gamma( \lambda_{\alpha} + \rho_{\alpha} - \frac{1}{4} m_{\alpha/2} - \frac{1}{2} m_{\alpha} + 1) }.
\]
We use the well known asymptotics of the $\Gamma$-function:
\[
 \frac{\Gamma(z+a)}{\Gamma(z+b)} \sim z^{a-b}.
\]
Then for all $\alpha \in \Sigma^+$ and $\lambda_{\alpha} \to \infty$ this implies the asymptotic
\begin{equation*}
f_{\alpha} (\la) \sim \la^{\frac{1}{2} m_{\alpha}} \cdot \la^{\frac{1}{2} m_{\alpha}} = \la^{m_{\alpha}},
\end{equation*}
Since $ \Sigma^+ $ is finite, we find a constant $M > 0$ such that for all positive roots
\begin{equation}\label{1}
|f_{\alpha} (\la) | \le 2 \la^{m_{\alpha}} \quad \text{ for } \la \ge M.
\end{equation}
Fix such $M$, and denote by $L>0$ a common upper bound such that
\begin{equation}\label{2}
|f_{\alpha} (\la)| \le L \quad \text{ for } \la < M \, \, (\forall \alpha \in \Sigma^+).
\end{equation}
Now let us temporarily fix a $\lambda \in \Lambda^+$. We decompose the set of positive roots in two disjoint sets $\Sigma^+ = \Sigma_1^+ \cup \Sigma_2^+$, where
\[
\Sigma_1^+ := \{ \alpha \in \Sigma^+ \, : \, \la < M \}, \quad \Sigma_2^+ := \{ \alpha \in \Sigma^+ \, : \, \la \ge M \}.
\]
Application of estimates ($\ref{1}$) and ($\ref{2}$) then yields:
\begin{equation*}
\left| \prod_{\alpha \in \Sigma^+} f_{\alpha} (\la) \right| \le \, L^{ | \Sigma_1^+ |} \cdot \prod_{\alpha \in \Sigma_2^+} 2 \la^{m_{\alpha}}.
\end{equation*}
Without loss of generality we may assume $L \ge 1$. Note that $\la \ne 0$ implies $\la \ge 1$ (since $\la \in \ZZ^+$). Therefore we can extend the above estimate:
\begin{equation*}
\left| \prod_{\alpha \in \Sigma^+} f_{\alpha} (\la) \right| \le L^{ | \Sigma^+ |} \cdot \prod_{\alpha \in \Sigma^+, \, \la \ne 0} 2\la^{m_{\alpha}}.
\end{equation*}
This holds independently of $\lambda$ and implies the lemma.
\end{proof}

The consequence of this lemma is that the growth of the summands in the heat kernel $\Gamma_m$ for $\lambda \to \infty$ is dominated by $e^{- \theta_\lambda t}$, which decays for fixed $t>0$ as $e^{- |\lambda|^2t}$. With \eqref{Schranke_Summanden} we conclude

\begin{Prop}\label{Schranke_Gamma}
The series defining the heat kernel $\Gamma_m$ converges absolutely and uniformly on $\liea\times \liea \times (0, \infty)$. For all $x,y \in \liea$ and  $t >0,$ we have
\[
|\Gamma_m (x,y,t) | \le \sum_{\lambda \in \Lambda^+} r_{\lambda} e^{- \theta_{\lambda} t} =: C_t < \infty.
\]
The long-time behaviour of $\Gamma_m$ is given by
\[
\lim_{t\to\infty} \Gamma_m(x,y,t) = r_0 = \frac{1}{\int_{\alcove} w_m(x) dx}
\] 
where the convergence is uniform on $\liea \times \liea.$
\end{Prop}

One would expect from classical theory of the heat equation that the heat kernel is smooth. This is also true in our setting.

\begin{Prop}\label{Kern_holomorph}
For fixed $t_0 > 0$ the heat kernel $\Gamma_m(\cdot, \cdot,t_0)$ extends to a holomorphic function on $\liea_{\CC} \times \liea_{\CC}$ which is $\Waff$-invariant in the real part of both arguments. The holomorphic extension is given by
\[
\Gamma_m(z,w,t_0) = \sum_{\lambda \in \Lambda^+} r_{\lambda} e^{- \theta_\lambda t_0} R_{\lambda}(z) R_{\lambda} (-w).
\]
In particular, $\, \Gamma_m\in C^{\infty}(\alcove\times\alcove \times (0, \infty))$.
\end{Prop}

\begin{proof}
It is obvious that $R_{\lambda}$ is holomorphic on $\liea_{\CC}$. Therefore each summand in $\,F(z,w) := \sum_{\lambda \in \Lambda^+} r_{\lambda} e^{- \theta_{\lambda} t} R_{\lambda}(z) R_{\lambda} (-w)\,$ is holomorphic, and normal convergence of the series will imply that $F$ is holomorphic on $\liea_{\CC}\times \liea_{\CC}.$ To see this, recall that the Jacobi polynomial $R_{\lambda}$ is a linear combination of exponentials $e^{i \gamma}$ with $\gamma \preceq \lambda$ according to Lemma \ref{Ordnung_W} (i). Part (ii) of Lemma \ref{Ordnung_W} then implies $|\gamma| \le |\lambda|$. Therefore
\[
\left| e^{i \langle \gamma, z \rangle} \right| \le e^{|\gamma| | \Im z|} \le e^{|\lambda| |z|}.
\]

For $M>0$, consider the compact ball $K:= \{z \in \liea_{\CC} : \, |z| \le M\}$. Then for $z \in K$ we obtain the following estimate:
\begin{align*}
|R_{\lambda} (z)| = \left| \sum_{\substack{\gamma \in W.\mu \\ \mu \in \Lambda^+, \; \mu \preceq \lambda}} d_{\lambda \gamma} e^{i \langle \gamma, z \rangle} \right| \le e^{|\lambda| | z|} \le e^{|\lambda| M }.
\end{align*}
Here we used $d_{\lambda \gamma} \ge 0$ and $\sum_{\gamma} d_{\lambda \gamma} = 1$ (see Remark \ref{Norm_Jacobi} (b)). With a similar estimate for $R_\lambda(w)$ we see that  the growth behaviour of each summand of $F$ is dominated by the term $e^{- \theta_\lambda t_0}$ on $K\times K$. Thus we  have normal convergence
on compact subsets of $\liea_{\CC}\times \liea_{\CC}$, which implies that $F$ is holomorphic on $\liea_{\CC}\times \liea_{\CC}$.
Finally note that termwise differentiating with respect to $t$ gives a factor $- \theta_{\lambda}$ but does not change the convergence.
\end{proof}

Next we collect some further basic properties of the heat kernel:
\begin{La}\label{Eig_Gamma}
\begin{enumeratea}
 \item For all $w\in \liea_{\CC}$ the function $u(x,t) := \Gamma_m (x,w,t)$ is a solution of the heat equation $L_m u = \partial_t u$ on $\alcove \times (0, \infty)$.
\item $\displaystyle \int_{\alcove} \Gamma_m(z,x,t) w_m (x) dx = 1 \quad (\forall z\in \liea_{\CC}).$
\item $\displaystyle \Gamma_m (z,w,t+s) = \int_{\alcove} \Gamma_m(z,x,t) \Gamma_m(x,w,s) w_m(x) dx \quad (\forall z,w, \in \liea_{\CC}).$
\item $\displaystyle \int_{\alcove} \Gamma_m(z,x,t) R_{\lambda} (x) w_m(x) dx = e^{- \theta_{\lambda} t} R_{\lambda}(z) \quad(\forall z\in \liea_{\CC}).$
\end{enumeratea}
\end{La}

\begin{proof}
Part (a) follows from \eqref{eigenequation} and termwise differentiation. The further statements are obtained by direct calculation, using the orthogonality of the polynomials $R_\lambda$ with respect to $\langle \cdot, \cdot \rangle_m$, and the fact that $R_\lambda(-x) = \overline{R_\lambda(x)}$ for all $x\in \overline{A_0}.$
\end{proof}

\begin{Def}
For $f \in L^1(\alcove, w_m)$ we define
\begin{equation}\label{heatdef}
 H(t) f(x) := \begin{cases} \int_{\alcove} \Gamma_m(x,y,t) f(y) w_m(y) dy & \, \text{ for } \, t > 0; \\ f(x) & \, \text{ for } \, t =0. \end{cases}
 \end{equation}
\end{Def}

With the Heckman-Opdam transform, we can write
\begin{equation*}
 H(t) f (x) = \sum_{\lambda \in \Lambda^+} r_{\lambda} e^{-\theta_{\lambda}t} \widehat f(\lambda) R_{\lambda} (x), \quad t > 0.
\end{equation*}
Because of $|\widehat f (\lambda)| \le \|f\|_1$ and Proposition \ref{Schranke_Gamma} this sum converges
absolutely and uniformly on $\alcove.$

\begin{La}\label{Schranke_H_tf_1}
Let $f \in L^1(\alcove,w_m)$. Then for each $t >0,$ we have $\,H(t)f\in C(\alcove)$ with
\[
\| H(t) f \|_{\infty} \le C_t \| f \|_1.\]
Moreover,
\[
H(t+s) f = H(t) H(s) f \quad (\forall s,t\geq 0).
\]
\end{La}

\begin{proof}
This follows directly from Proposition  \ref{Schranke_Gamma} and Lemma \ref{Eig_Gamma} (c).
\end{proof}

Since $\alcove$ is compact we have continuous embeddings
\begin{equation}\label{imbed}
C(\alcove) \hookrightarrow L^{\infty} (\alcove,w_m) \hookrightarrow L^p (\alcove, w_m) \hookrightarrow L^1 (\alcove, w_m), \quad 1 < p < \infty.
\end{equation}
In particular, Lemma \ref{Schranke_H_tf_1} implies that the family $(H(t))_{t\geq 0}$ forms a semigroup of bounded linear operators on the Banach space $(C(\alcove), \|.\|_\infty).$
We shall prove that this semigroup is actually a Feller-Markov semigroup on $C(\alcove)$ which is generated
by the closure of the Heckman-Opdam Laplacian $L_m$. Feller-Markov means that the semigroup is strongly continuous, contractive and positive, i.e. $f\geq 0$ on $\alcove$ implies $H(t)f\geq 0$ on $\alcove$.
As for the rational Dunkl case and the noncompact trigonometric case in \cite{Ro98} and \cite{Scha}, the proof of the positivity part will be based on a variant of the Lumer-Phillips theorem characterizing the generator of a Feller-Markov semigroup, c.f. Theorem 2.2. of \cite{EK86}.

\begin{Thm}
The family $(H(t))_{t \ge 0}$ is a Feller-Markov semigroup on $(C(\alcove), \| .\|_{\infty})$. Its generator is given by the closure $\overline L_m$ of $L_m$.
\end{Thm}

We call $(H(t))_{t\geq 0}$ the \emph{Heckman-Opdam heat semigroup on $\alcove$.}

\begin{proof} We first prove strong continuity and determine the generator. For this, notice that
\begin{equation}\label{H_tR}
H(t) R_\lambda = e^{- \theta_{\lambda} t} R_{\lambda} \quad (\lambda\in \Lambda^+).
\end{equation}
This immediately implies strong continuity of the semigroup on the space $\mathcal T^W$ of $W$-invariant trigonometric polynomials. As $\mathcal T^W$ is dense in $C(\alcove)$, we obtain strong continuity on all of $C(\alcove).$
Let $A$ be the generator of $(H(t))_{t \ge 0}$. Since $\mathcal T^W$ is also $H(t)$-invariant it is by Nelson's Lemma (Theorem 6.1.18, \cite{Davies}) a core for the generator $A$. We calculate
\[\lim_{t \to 0} \frac{H(t) R_{\lambda} - R_{\lambda}}{t} \,=\,\lim_{t \to 0} \frac{e^{ - \theta_{\lambda} t} - 1}{t} R_{\lambda} \,=\, - \theta_{\lambda} R_{\lambda}\, =\, L_m R_{\lambda}.
\]
This shows that $\,A\vert _{\mathcal T^W} =L_m$. It remains to prove that the semigroup $(H(t))_{t\geq 0}$ is Feller-Markov. We shall apply Theorem 2.2. of \cite{EK86} where we consider $\CC$-valued functions on $\alcove$.
We thus have to check the following three conditions:

\smallskip
\begin{enumeratei}
\item If $f \in \mathcal D (L_m) = \mathcal T^W$ then also $\overline f \in \mathcal D(L_m)$ and $L_m (\overline f) = \overline{L_m (f)}$.
\item There exists a $t > 0$ such that the range of $t I - L_m$ is dense in $C(\alcove)$.
\item If $f \in \mathcal D (L_m)$ is real-valued with a nonnegative maximum in $x_0 \in \alcove$, i.e. $0 \leq f(x_0) = \textnormal{max}_{x \in \alcove} f(x)$, then $L_m f (x_0) \leq 0$ (Positive maximum principle).
\end{enumeratei}

\smallskip
Condition (i) is obvious. Condition (ii) is also clear because we already know that $(H(t))_{t\geq 0}$ is a strongly continuous semigroup and therefore $t >0$ is contained in the resolvent set of $A$ for $t$ large enough. The positive maximum principle is obvious when $x_0 \not \in H_{\alpha}$ for all $\alpha \in \Sigma$. If $\langle \alpha, x_0 \rangle =0$ for some $\alpha \in \Sigma$ one has to use a similar argument as in the proof of Lemma 4.1 in \cite{Ro98}: Consider $f$ as a $\Waff$-invariant function on $\RR^q$ and let $x \not \in H_{\alpha}$ for all $\alpha \in \Sigma$. Then Taylor expansion yields
\[
0 = f( s_{\alpha} x ) - f(x) = - \langle \alpha, x \rangle \partial_\alpha f(x) + \frac 1 2 \langle \alpha, x \rangle^2 \alpha^T D^2f(\xi) \alpha,
\]
where $\xi$ lies on the line segment between $x$ and $s_{\alpha}x$. Therefore
\[\lim_{x\to x_0} \cot \langle \alpha, x \rangle \partial_\alpha f(x)= \frac 1 2 \alpha^T D^2f(x_0) \alpha \le 0.\]
\end{proof}

The positivity of the heat semigroup implies that $\Gamma_m$ is non-negative on the alcove $\alcove$. In fact we have more:

\begin{Prop}
The heat kernel $\Gamma_m$ is strictly positive, i.e.
\[
\Gamma_m(x,y,t) > 0 \quad \text{ for all } (x,y,t) \in \alcove \times \alcove \times (0, \infty).
\]
\end{Prop}

\begin{proof}
Assume that $\Gamma_m(x_0,y_0, t_0) =0$. According to Lemma \ref{Eig_Gamma} (c) we have
\[
\Gamma_m(x_0,y_0, t_0) = \int_{\alcove} \Gamma_m \left(x_0,a, \frac{t_0} 2 \right) \, \Gamma_m \left(a, y_0, \frac{t_0} 2 \right) w_m(a) da.
\]
Let $F_{t} (z,w) := \Gamma_m(z,w,t) = \sum_{\lambda \in \Lambda^+} r_{\lambda} e^{- \theta_{\lambda} t} R_{\lambda} (z) R_{\lambda} (-w)$. Now the positivity of the heat semigroup implies $\Gamma_m \ge 0$. Therefore
\[
 F_{\frac{t_0} 2}(x_0,a) F_{\frac{t_0} 2}(a, y_0) \equiv 0 \quad \text{ on } \alcove.
\]
But we already know from Proposition \ref{Kern_holomorph} that $F_t$ is holomorphic in both arguments. If the product of two holomorphic functions vanishes on an open, connected and nonempty set, then one of them has to be identically zero. But this is a contradiction to $F_t (x_0,x_0) = \sum r_{\lambda} e^{- \theta_{\lambda} t} |R_{\lambda} (x_0)|^2 >0$.
\end{proof}

%
%
Let us conclude this section with some remarks concerning $L^p$-theory and behaviour for $t \to \infty$: So far we have considered the heat semigroup on the space of continuous functions $C(\alcove)$. But an $L^p(\alcove, w_m)$-theory ($1\le p \le \infty$) is easily developed:

By Lemma \ref{Schranke_H_tf_1} and the embeddings \eqref{imbed}, each $H(t)$ defines a bounded linear operator on $L^p(\alcove, w_m)$ for $1\leq p\leq \infty$.

\begin{Prop}
The family $(H(t))_{t\geq 0}$ defines a positive, contractive and strongly continuous semigroup on $L^p(\alcove, w_m)$ for $1\leq p \leq \infty$.
\end{Prop}

\begin{proof}
The positivity of $\Gamma_m$ implies that $H(t)$ is positive on $L^p(\alcove, w_m)$. Moreover, if $1\leq p <\infty$ then Jensen's inequality implies
\[ |H(t)f(x)|^p \leq \int_{\alcove} |f(y)|^p\, \Gamma_m(x,y,t) w_m(y)dy\]
and therefore $\|H(t)f\|_p \leq \|f\|_p$. For $p=\infty$, this estimate is obvious.
It remains to check strong continuity of the semigroup on the dense subspace $\mathcal T^W$. But as in
the case of $C(\alcove)$, this is immediate from \eqref{H_tR}.
\end{proof}

\begin{Rem}
The $L^2$-theory of the heat semigroup is particularly  explicit. The Laplacian $L_m$ with domain
$\mathcal T^W$ is symmetric in $L^2(T,w_m)$ (\cite{Op95}, Proposition 2.3), and the  Jacobi polynomials $R_\lambda$ form a complete set of eigenfunctions with real eigenvalues. Therefore $L_m$ is essentially self-adjoint.
 Its closure is given by orthogonal expansion with respect to the Jacobi basis:
\[
\overline{L_m} f = \sum_{\lambda \in \Lambda^+} r_{\lambda} \theta_{\lambda} \langle f, R_{\lambda} \rangle_m R_{\lambda} = \sum_{\lambda \in \Lambda^+} r_{\lambda} \theta_{\lambda} \widehat f(\lambda) R_{\lambda}
\]
with domain
\[
\mathcal D (\overline{L_m}) = \{f \in L^2(\alcove, w_m) \, : \, \sum_{\lambda \in \Lambda^+} r_{\lambda} \theta_{\lambda}^2 |\widehat f(\lambda)|^2 < \infty \}.
\]
This self-adjoint operator generates a strongly continuous semigroup $e^{t \overline{L_m}}$ on $L^2(\alcove, w_m).$ By  Borel functional calculus,
\[
e^{t \overline{L_m}} R_{\lambda} = e^{- \theta_{\lambda} t} R_{\lambda}
\]
and for  general $f= \sum_{\lambda} r_{\lambda} \widehat f(\lambda) R_{\lambda},$
\[
e^{t \overline{L_m}} f = \sum_{\lambda \in \Lambda^+} r_{\lambda} e^{- \theta_{\lambda} t} \widehat f(\lambda) R_{\lambda}.
\]
This coincides with the heat semigroup $(H(t))_{t \ge 0}$ on $L^2(\alcove, w_m)$.
\end{Rem}

For $t \to \infty$ the heat (given by the initial distribution $f$) spreads uniformly on the alcove:

\begin{Prop}
Let $f \in L^p(\alcove, w_m), \,1\leq p \leq \infty.$  Then
\[
 \lim_{t \to \infty} H(t) f = \frac{1}{\int_{\alcove} w_m(x) dx} \int_{\alcove} f(x) \, w_m(x)dx.
\]
with respect to $\| . \|_p$.
\end{Prop}

\begin{proof}
Write
\[
H(t) f (x) = r_0 \widehat f(0) + \sum_{\lambda \in \Lambda^+, \lambda \ne 0} r_{\lambda} e^{-\theta_{\lambda}t} \widehat f(\lambda) R_{\lambda} (x).
\]
and take the limit $t \to \infty$.
\end{proof}

%
\section{The Segal-Bargmann transform}\label{sec_SB}
%
In this section we focus on the $L^2$-setting. The smoothness of the heat kernel implies that the heat transform $f\mapsto H(t)f$  smoothens arbitrary initial data.  We shall see that it actually gives rise to a unitary isomorphism between $L^2(\alcove, w_m)$ and a certain Hilbert space $\mathcal H_t$ of holomorphic functions on $\liea_{\CC}$ - the so called \emph{Segal-Bargmann transform}.


Let us start with a short reminder of the classical situation for the one-dimensional torus $\mathbb T = \RR / 2\pi \ZZ$; for details see e.g. \cite{Faraut}.

\begin{Ex} (The Segal-Bargmann transform for the torus $\mathbb T = \RR/2\pi\ZZ$).
We consider functions on $\mathbb
T$ as functions on $\RR$ which are invariant under the action of $2 \pi
\ZZ$. Let $f \in L^2(\mathbb T)$. The heat equation
\[
\Delta u = \partial_t u, \quad u(x,0) = f(x)
\]
has the solution
\[
u(x,t) = H(t)f (x) := \sum_{n \in \ZZ} \widehat f(n) e^{-n^2 t} e^{inx}.
\]
with the usual Fourier coefficients
\[
\widehat f (n) = \frac{1}{\sqrt{2 \pi}} \int_{\mathbb T} f(x) e^{inx} \, dx.
\]
The function $H(t)f$ extends holomorphically to a function on $\CC$
which is $2 \pi \ZZ$-periodic in the real part of its argument:
\[
H(t)f (z) = \sum_{n \in \ZZ} \widehat f(n) e^{- n^2 t} e^{i
nz}.
\]
Consider the heat kernel on $\RR$,
\[
\gamma^1_t (y) = \frac{1}{\sqrt{4 \pi t}}\, e^{- |y|^2 / 4t}
\]
and put
\[
\rho_t (y) := 2 \gamma^1_{2t} (2y).
\]
Then the image of
\[
H(t) : L^2(\mathbb T) \to \mathcal O(\CC/ 2 \pi \ZZ)
\]
is the Segal-Bargmann space
\[
\mathcal H_t := \{ F \in \mathcal O(\CC/ 2 \pi \ZZ) : \, \| F
\|_{\mathcal H_t} := \frac{1}{\sqrt{2 \pi}} \int_{\mathbb T}
\int_{\RR} |F(x+iy)|^2 \rho_t (y) dx dy \le \infty \}.
\]
Moreover, the Segal-Bargmann transform $H(t): L^2(\mathbb T) \to \mathcal H_t\,$
is a unitary isomorphism.
\end{Ex}

 We shall now extend this result to the compact (and Weyl-group invariant) Heckman-Opdam case.
 The classical example as well as the theory in the case of compact symmetric spaces indicates that the Segal-Bargmann space will depend on the noncompact heat kernel (see \cite{Faraut}).

%
%
\begin{Rem} (The noncompact heat equation) \\
The noncompact Heckman-Opdam heat equation was studied by Schapira \cite{Scha}. Let us recall some results:

Denote by $D_m$ the $W$-invariant part of $\Delta_m,$ i.e.
\[
D_m = \Delta + \sum_{\alpha \in \Sigma^+} m_{\alpha} \coth \langle \alpha, x \rangle \partial_{\alpha}.
\]
The hypergeometric function $F_{\lambda}$ is an eigenfunction of the operator $D_m$:
\begin{equation}\label{F_eigenwert}
D_m F_{\lambda + \rho} = \langle \lambda, \lambda + 2\rho \rangle F_{\lambda + \rho} = \theta_{\lambda} F_{\lambda + \rho}.
\end{equation}
The operator $D_m$ has a closure which generates the noncompact heat semigroup $e^{t \overline{D_m}}$ on $C_0(\liea)$. The $W$-invariant noncompact heat kernel is given by
\begin{equation*}
 \gamma^1_t (x,y) = \int_{i \liea} e^{- t (|\lambda|^2 + |\rho|^2)} F_{\lambda}(x) F_{\lambda} (-y) \frac{d \lambda}{|c(\lambda)|^2}.
\end{equation*}
We will use the notation $\gamma^1_t (x) := \gamma^1_t (x,0)$ in the following. From $\eqref{F_eigenwert}$ we see that
\begin{equation}\label{nk}
\int_{\liea} F_{\lambda+\rho} (x) \gamma^1_{t} (x) \delta_m(x) dx = \left(e^{t\overline{D_m}} F_{\lambda+\rho} \right) (0) = e^{t \theta_{\lambda}} F_{\lambda+\rho} (0) = e^{t \theta_{\lambda}}
\end{equation}
where
\[ \delta_m:= \prod_{\alpha \in \Sigma^+} \left| e^{\alpha} - e^{-\alpha} \right|^{m_{\alpha}}\]
 is Opdam's weight function.
\end{Rem}

We now turn to the holomorphic extension of the heat transform on $L^2(\alcove, w_m).$
Recall that the heat kernel
\[ \Gamma_m(z,w,t) = \sum_{\lambda \in \Lambda^+} r_{\lambda} e^{- \theta_\lambda t} R_{\lambda}(z) R_{\lambda} (-w).
\]
is holomorphic in $(z, w)$, where the series converges normally on compact subsets of $\liea_{\CC}\times \liea_{\CC}.$ Thus for $f \in L^2(\alcove, w_m)$ and $t > 0$, the heat transform $H(t)f$ given by \eqref{heatdef}
extends to a holomorphic function on $\liea_{\CC}$,
\begin{equation*}
 H(t) f (z) = \left(e^{t \overline{L_m}} f \right) (z) = \int_{\alcove} \Gamma_m(z,y,t) f(y) w_m(y) dy.
\end{equation*}
Alternatively, this can be written as
\begin{equation}\label{H(t)_f(z)}
 H(t) f(z) = \sum_{\lambda \in \Lambda^+} r_{\lambda} e^{- \langle \lambda, \lambda +2\rho \rangle t} \widehat f(\lambda) R_{\lambda} (z)
\end{equation}
where the sum converges normally on compact subsets of $\liea_{\CC}.$

\begin{Def}
In the following ${H(t)}$ shall always stand for the analytic continuation of the heat transform to $\liea_{\CC}$.
\end{Def}

\begin{Def}
 Define $\mathcal F_t := \text{Im} (H(t)) \subset \mathcal O(\liea_{\CC})$ as the holomorphic image of the heat transform with inner product
\[
\langle H(t) f, H(t) g \rangle_{\mathcal F_t} := \langle f,g \rangle_{L^2(\alcove, w_m)}.
\]
\end{Def}

\begin{Prop}
 The space $(\mathcal F_t, \| . \|_t)$ is a Hilbert space with reproducing kernel
\[
K_t (z,w) = \Gamma_m (w, \overline z, 2t).
\]
The set of $W$-invariant trigonometric polynomials $\mathcal T^W$ is dense in $(\mathcal F_t, \| \cdot \|_{\mathcal F_t})$.
\end{Prop}

\begin{proof}
The heat transform $L^2(\alcove, w_m) \to \mathcal F_t$ is by definition a unitary isomorphism. The set $\mathcal T^W$ is dense in $L^2(\alcove, w_m)$ and therefore also in $\mathcal F_t$ (recall that $H(t)$ maps $\mathcal T^W$
onto itself).
Finally let $F = H(t) f \in \mathcal F_t$. Then
\begin{align*}
F(z) & = H(t) f(z) = \int_{\alcove} f(y) \Gamma_m(z,y,t) w_m(y) \, dy \\
& = \int_{\alcove} f(y) \sum_{\lambda} r_{\lambda} e^{- t \theta_{\lambda}} R_{\lambda} (z) R_{\lambda} (-y) w_m(y) \, dy \\
& = \int_{\alcove} f(y) \sum_{\lambda} r_{\lambda} e^{- t \theta_{\lambda}} \overline{R_{\lambda} (\overline y) R_{\lambda} (- \overline z) } w_m(y) \, dy \\ & = \langle f, \Gamma_m(\overline \cdot , \overline z, t) \rangle_{L^2(\alcove, w_m)} = \langle H(t) f, H(t) \Gamma_m (\overline \cdot, \overline z, t) \rangle_{\mathcal F_t} = \langle F, K_{t,z} \rangle_{\mathcal F_t}
\end{align*}
with
\begin{align*}
K_{t,z} (w) & = H(t) \Gamma_m (\overline \cdot, \overline z, t) (w) = \int_{\alcove} \Gamma_m(w,y,t) \Gamma_m( \overline y, \overline z,t) w_m(y) \, dy \\
& = \Gamma_m(w, \overline z, 2t)
\end{align*}
according to Lemma \ref{Eig_Gamma} (c).
\end{proof}


We are interested in a more explicit description of the image of $H(t)$ as a Hilbert space of holomorphic functions.

%
%
\begin{Def}
Assume that the Fourier coefficients of $f \in L^2(\alcove, w_m)$ satisfy the growth condition
\begin{equation}\label{growth_cond}
\sum_{\lambda \in \Lambda^+} r_{\lambda} |\widehat f (\lambda)|^2 e^{2 |\lambda+\rho| |x|} < \infty \quad (\forall x \in \liea).
\end{equation}
For such $f$ and $x, y \in \liea$ we define a generalized translation by
\[
 \tau_{ix} f(y) := f (- ix \ast y) := \sum_{\lambda \in \Lambda^+} r_{\lambda} \widehat f (\lambda) R_{\lambda} (y) R_{\lambda}(-ix).
\]
\end{Def}
Recalling that $\, R_\lambda(-ix) = F_{\lambda+\rho}(x)$, we
observe that $\tau_{ix} f \in L^2(\alcove, w_m)$ if and only if
\begin{equation}\label{L2cond}
\| \tau_{ix} f \|_2^2 < \infty \iff \sum_{\lambda \in \Lambda^+} r_{\lambda} |\widehat f (\lambda)|^2 |F_{\lambda + \rho} (x)|^2 < \infty.
\end{equation}
Notice also that $x\mapsto \tau_{ix}f(y)$ is $W$-invariant on $\liea$.
According to Proposition 6.1 in \cite{Op95}, the hypergeometric function satisfies a growth estimate
\[|F_{\lambda} (x) | \le C e^{|\lambda| |x|} \quad \forall x\in \liea\]
 with a constant $C$ independent of $x$ and $\lambda$. This shows that condition \eqref{L2cond} above is implied by our growth condition $\eqref{growth_cond}$, and the above translation is indeed well-defined.
Note also that
\begin{equation*}
R_{\lambda} ( -ix \ast y ) = R_{\lambda} (-ix) R_{\lambda} (y).
\end{equation*}

Next we define the target space of the Segal-Bargmann transform:

\begin{Def} Let
\[
\mathcal H_t := \{ F \in \mathcal O(\liea_{\CC}) \, : \, \, F \text{ is } \Waff\text{-invariant in the real part of its argument;} \, \| F \|_{\mathcal H_t} < \infty \},
\]
with the inner product
\[
\langle F, G \rangle_{\mathcal H_t} := \sum_{\lambda \in \Lambda^+} r_{\lambda} \widehat F(\lambda) \overline{\widehat G(\lambda)} e^{2t \theta_{\lambda}}.
\]
Here the (Heckman-Opdam) Fourier transform $\,\widehat F(\lambda) := \int_{\alcove} F(x) R_{\lambda} (-x) w_m(x) dx\,$ is with respect to real part $x$ of the variable.
\end{Def}

Note that functions in $\mathcal H_t$ automatically satisfy growth condition $\eqref{growth_cond}$ (with respect to the real part of the variable). Moreover, the identity theorem implies that each $F\in \mathcal H_t$ is
invariant under the action of $W$ on $\liea_{\CC}$, i.e. it is $W$-invariant also in the imaginary part of its argument.

\begin{Prop}\label{Prop_HHF}
The space $\mathcal H_t$ is a Hilbert space of holomorphic functions.
\end{Prop}

\begin{proof}
Let us recall the definition: A \emph{Hilbert space of holomorphic functions} $\mathcal H$ on a domain $D$ is a subspace of $\mathcal O(D)$ with the structure of a Hilbert space such that the embedding $\mathcal H \hookrightarrow \mathcal O(D)$ is continuous.

This means, for every compact subset $K \subset \liea_{\CC}$ we have to find a constant $C_K$ such that
\[
|F(z) | \le C_K \| F \|_{\mathcal H_t} \quad \forall z \in K.
\]
We claim that each $F \in \mathcal H_t$ can be written as
\begin{equation}\label{F_Entw}
F(z) = \sum_{\lambda \in \Lambda^+} r_{\lambda} \widehat F(\lambda) R_{\lambda} (z)
\end{equation}
where the series converges normally on each compact subset $K$ of $\liea_{\CC}.$ Indeed, let $z\in K$ with $| \Im z| \le M$. Then
\begin{align*}
\sum \big| r_{\lambda} \widehat F(\lambda) R_{\lambda} (z) \big| & \le \sum r_{\lambda} | \widehat F(\lambda)| e^{t\theta_{\lambda} } e^{-t\theta_{\lambda} }|R_{\lambda} (z)| \\
& \le \left( \sum r_{\lambda} | \widehat F(\lambda)|^2 e^{2 t\theta_{\lambda} } \right)^{1/2} \cdot \left( \sum r_{\lambda} \left| R_{\lambda} (z) \right|^2 e^{-2 t\theta_{\lambda} } \right)^{1/2}.
\end{align*}
Now the first factor is just $\| F \|_{\mathcal H_t}$ and for the second factor we use the estimate $| R_{\lambda} (z) |^2 \le e^{2|\lambda| |\Im z|} \le e^{ 2M |\lambda|}$ to obtain a constant $C_K$ such that
\begin{equation}\label{estimate_Entw}
\sum \| r_{\lambda} \widehat F(\lambda) R_{\lambda}\|_{\infty, K}\,\le C_K \| F \|_{\mathcal H_t}.
\end{equation}
As a consequence, the sum in \eqref{F_Entw} defines a holomorphic function on $\liea_{\CC}$. On the other hand, for given $f\in L^2(\alcove)$ the sum $\,\sum_{\lambda \in \Lambda^+} r_{\lambda} \widehat f(\lambda) R_{\lambda}$ is just the expansion of $f$ with respect to the orthonormal basis of Heckman-Opdam polynomials. Since $F \in \mathcal H_t$ is continuous and bounded in the real part (as a $\Waff$-invariant function) we have $F \in L^2(\alcove)$ and therefore $\,F(x) = \sum_{\lambda \in \Lambda^+} r_{\lambda} \widehat F(\lambda) R_{\lambda} (x)$ a.e. on $\alcove$. Now the identity theorem implies the claim $\eqref{F_Entw}$.

Relation $\eqref{F_Entw}$ together with $\eqref{estimate_Entw}$ show that the embedding $\mathcal H_t \hookrightarrow \mathcal O(\liea_{\CC})$ is continuous. It remains to check that $\mathcal H_t$ is complete with respect to the given inner product $\langle \cdot, \cdot \rangle_{\mathcal H_t}.$ For this, consider a Cauchy sequence $(F_n)$ in $\mathcal H_t$. Because of the continuous embedding it converges uniformly on compact subsets of $\liea_{\CC}$ to some limit $F\in\mathcal O( \liea_{\CC})$ which is $\Waff$-periodic in its real part. In particular $(F_n)$ converges uniformly on the alcove $\alcove$. This implies that $\,\widehat {F_n}(\lambda) \to \widehat F(\lambda)$ uniformly in $\lambda$, and
therefore $\lim_{n\to\infty} F_n = F$ in $\mathcal H_t$.
\end{proof}

In the next Proposition we give another representation of the inner product of $\mathcal H_t$ and state its reproducing kernel.

\begin{Prop} \begin{enumerate}
              \item[\rm(1)]
The inner product of $\mathcal H_t$ can be written as
\begin{equation}\label{H_T_alternativ}
\langle F, G \rangle_{\mathcal H_t} = \int_{\liea} \int_{\alcove} \tau_{ix} F(y)\, \overline{G (y)} w_m(y) \gamma^1_{2t} (x) \delta_m(x) \, dy dx.
\end{equation}
\item[\rm(2)] The reproducing kernel of $\mathcal H_t$ is given by $\, K_t (z,w) = \Gamma_m (w, \overline z, 2t)$.
\end{enumerate}
\end{Prop}

We remark that in \eqref{H_T_alternativ}, $\,\int_{\liea} $ can be replaced by $\,|W|\int_{\liea^+}\, $ because  
$x\mapsto \tau_{ix}F(y)$ and the heat kernel $\gamma_{2t}^1$ are $W$-invariant.

\begin{proof}
We already remarked that $\tau_{ix} F$ is well defined for functions in $\mathcal H_t$. Using dominated convergence we calculate
\begin{align*}
 & \int_{\liea} \int_{\alcove} \sum_{\lambda} r_{\lambda} \widehat F(\lambda) R_{\lambda} (-ix) R_{\lambda} (y) \overline{G (y)} w_m(y) \gamma^1_{2t} (x) \delta_m(x) \, dy dx \\
= \, & \sum_{\lambda} r_{\lambda} \widehat F(\lambda) \int_{\liea} \left( \int_{\alcove} R_{\lambda} (y) \overline{G (y)} w_m(y) dy \right) R_{\lambda} (-ix) \gamma^1_{2t} (x) \delta_m(x) \, dx\\
= \, & \sum_{\lambda} r_{\lambda} \widehat F(\lambda) \overline{\widehat G(\lambda}) \int_{\liea} F_{\lambda+\rho} (x) \gamma^1_{2t} (x) \delta_m(x) \, dx = \sum_{\lambda} r_{\lambda} \widehat F(\lambda) \overline{\widehat G(\lambda}) e^{2t \theta_{\lambda}} = \langle F, G \rangle_{\mathcal H_t}.
\end{align*}
Here we used $R_{\lambda} (-ix) = F_{\lambda + \rho} (x)$ and $\eqref{nk}$. For the reproducing kernel property
note that for  $\, K_{t,z}(w) = \Gamma_m(w,\overline z, 2t)\, $ we have
\[
\widehat K_{t,z} (\lambda) = e^{-2t\theta_{\lambda}} \overline{R_{\lambda} (z)}.      
\]
Thus
\[
\langle F, K_{t,z} \rangle_{\mathcal H_t} = \sum_{\lambda \in \Lambda^+} r_{\lambda} \widehat F(\lambda) R_{\lambda} (z)
\]
which is equal to $F(z)$ according to $\eqref{F_Entw}$.
\end{proof}

So $\mathcal F_t$ and $\mathcal H_t$ have the same reproducing kernel. But then by general Hilbert space theory they have to coincide. The holomorphic heat transform $H(t) : L^2(\alcove, w_m) \to \mathcal F_t$ is by definition a unitary isomorphism. We conclude

\begin{Thm}
The Segal-Bargmann transform
\[
L^2(\alcove, w_m) \to \mathcal H_t, \quad f \mapsto H(t)f
\]
defined by $\eqref{H(t)_f(z)}$ is a unitary isomorphism from the $L^2$-space on the alcove $\alcove$ onto the Hilbert space of holomorphic functions $\mathcal H_t$.
\end{Thm}


\begin{thebibliography}{9999999}\itemsep=-3pt
\bibitem{SaidOrsted} S. Ben Sa\"{\i}d, B. \O{}rsted, \textit{Segal-Bargmann transforms associated with finite Coxeter groups}, Math. Ann. 334 (2006), no. 2, 281-323.
\bibitem{EK86} S. Ethier, T. Kurtz, \textit{Markov Processes: Characterization and Convergence}, John Wiley and Sons, Inc., New York, 1986.
\bibitem{Davies} B. Davies, \textit{Linear Operators and their Spectra}, Cambridge, 2007.
\bibitem{Faraut} J. Faraut, \textit{Espaces hilbertiens invariants de fonctions holomorphes}, S\'{e}min. Congr., 7, Soc. Math. France, 2003, 101-167.
\bibitem{Hall} B. Hall, \textit{The Segal-Bargmann ``Coherent State'' Transform for Compact Lie Groups}, J. Funct. Anal. 122 (1994), 103-151. 
\bibitem{HS94} G. Heckman and H. Schlichtkrull, \textit{Harmonic Analysis and Special Functions on Symmetric Spaces}, Perspectives in Mathematics, Vol. 16, Academic Press, 1994.
\bibitem{Hel84} S. Helgason, \textit{Groups and  geometric analysis}, Academic Press, 1984.
\bibitem{Humph} J. Humphreys, \textit{Introduction to Lie algebras and representation theory}, Springer, 1972.
\bibitem{Mac87} I. Macdonald, \textit{Orthogonal polynomials associated with root systems}, Preprint 1987; reproduced in: S\'{e}minaire Lotharingien de Combinatoire 45 (2000), Article B45a.
\bibitem{OS09} G. \'{O}lafsson and H. Schlichtkrull, \textit{Fourier transforms of spherical distributions on compact symmetric spaces}, preprint, arXiv:0810.0062.
\bibitem{Op95} E. Opdam, \textit{Harmonic Analysis for certain representations of graded Hecke algebras}, Acta Math. 175 (1995), 75-121.
\bibitem{Op00} E. Opdam, \textit{Lectures on Dunkl operators for real and complex reflection groups}, MSJ Memoirs 8, Math. Soc. of Japan, 2000.
\bibitem{OS07}G. \'{O}lafsson, H. Schlichtkrull, \textit{The Segal-Bargmann transform for the heat equation associated with root systems}, Adv. Math. 208 (2007), no. 1, 422-437.
\bibitem{Ro98} M. R\"osler, \textit{Generalized Hermite Polynomials and the Heat Equation for Dunkl Operators}, Commun. Math. Phys. 192 (1998), 519-541.
\bibitem{Sahi} S. Sahi, \textit{A new formula for weight multiplicities and characters}, Duke Math J. 101 (2000), 77-84.
\bibitem{Scha} B. Schapira, \textit{Contributions to the hypergeometric function theory of Heckman and Opdam: sharp estimates, Schwartz space, heat kernel}, Geom. Funct. Anal. 18 (2008), 222--250.
\bibitem{Soltani} F. Soltani, \textit{Generalized Fock spaces and Weyl commutation relations for the Dunkl kernel}, Pacific J. Math.  214  (2004),  379--397.
\bibitem{Sontz}  S.B. Sontz, \textit{On Segal-Bargmann analysis for finite Coxeter groups and its heat kernel}. ArXiv:0903.2284.
\bibitem{St99} M.L. Stenzel, \textit{The Segal Bargmann transform on symmetric spaces of compact type.}
J. Funct. Anal. 165 (1999), 44--58.

\end{thebibliography}
\end{document}